\documentclass[reqno,a4paper,11pt, twoside]{amsart}
\usepackage{amsfonts,amssymb,epsfig,amstext,amsthm,mathpazo}
\usepackage[english]{babel}
\usepackage{geometry}
\usepackage{graphics}
\usepackage{subcaption}
\usepackage[font=footnotesize,labelfont=bf]{caption}
\usepackage[labelsep=period]{caption}
\usepackage{setspace}
\usepackage{enumerate}
\usepackage{latexsym, graphicx}
\usepackage{float}
\newfont{\bb}{msbm10 at 12pt}

\usepackage{fancyhdr}
\newcounter{num} 
\newcommand{\Fg}[1][]{\thenum}
\newtheorem{defi}{Definition}[section]
\newtheorem{lem}{Lemma}[section]

\newtheorem{thm}{Theorem}[section]

\newtheorem{rem}[thm]{Remark}

\numberwithin{equation}{section}

\begin{document}

\title{$\ast$-conformal Einstein Solitons on $\mathcal{N}(k)$-contact metric Manifolds}
\author{Jhantu Das, Kalyan Halder, Soumendu Roy and Arindam Bhattacharyya}
\address{Department of Mathematics, Sidho-Kanho-Birsha University, Purulia, West Bengal-723104, India.\newline \indent Division of Mathematics, School of Advanced Sciences, Vellore Institute of Technology, Chennai-600127, India\newline
\indent Department of Mathematics, Jadavpur University, Kolkata, West Bengal-700032, India.\newline}
\email{dasjhantu54@gmail.com, drkalyanhalder@gmail.com, soumendu.roy@vit.ac.in,\newline bhattachar1968@yahoo.co.in}


\maketitle

\begin{abstract} 
The main goal of this paper is devoted to $\mathcal{N}(k)$-contact metric manifolds admitting $\ast$-conformal Einstein soliton and also $\ast$-conformal gradient Einstein soliton. {In this settings the nature of the manifold, and the potential vector field, potential function of solitons are characterized, and conditions for the $\ast$-conformal Einstein soliton to be expanding, steady, or shrinking are also given.} Furthermore, the nature of the potential vector field is evolved when the metric $g$ of $\mathcal{N}(k)$-contact metric manifold satisfies $\ast$-conformal gradient Einstein soliton. Finally, an illustrative example of a $\mathcal{N}(k)$-contact metric manifold is discussed to verify our findings.  
\end{abstract}

\newcommand{\E}{\operatorname{E}}
\newcommand{\var}{\operatorname{var}}
\newcommand{\sd}{\operatorname{sd}}
\newcommand{\cov}{\operatorname{cov}}
\newcommand{\diff}{\mathop{}\!\mathrm{d}}

\section{Introduction}\label{sec1}
\quad A Riemannian metric $g$ defined on a smooth manifold $\mathcal{M}$ is called an Einstein soliton if there exists a real constant $\Omega$ and a smooth vector field $\mathcal{V}$ on $\mathcal{M}$ such that the following equation holds, \cite{4}  

\begin{equation}\label{1.1} Ric+\frac{1}{2}\pounds_{\mathcal{V}}g+(\Omega-\frac{r}{2})g=0,\end{equation}\\where $\pounds_{\mathcal{V}}g$ denotes the Lie derivative of the metric $g$ along the direction of $\mathcal{V}$ and $Ric$, $r$ are the symmetric Ricci tensor of type $(0,2)$, and scalar curvature of the smooth manifold $\mathcal{M}$, respectively. Naturally,  $\mathcal{V}$ and $\Omega$ are said to be the potential vector field and the soliton constant, respectively. The potential vector field and the soliton constant of the soliton play vital roles in determining the nature of the soliton. An Einstein soliton is said to be shrinking if $\Omega<0$, steady if $\Omega=0$ and expanding if $\Omega>0$. Moreover if $\mathcal{V}$ is of gradient type i.e., $\mathcal{V}=grad(f)$, for some real valued smooth function $f$ defined on $\mathcal{M}$, then the soliton (\ref{1.1}) is called gradient Einstein soliton \cite{4}. The Einstein flow is an evolution equation for metrics $g$ on a Riemannian manifold $\mathcal{M}$ defined by \begin{align*}\frac{\partial}{\partial{t}}({g(t)})+2Ric(g(t))=rg(t),\end{align*} where $g(t)$ is a one-parameter family of metrics on $\mathcal{M}$. 
Einstein solitons are self-similar solutions to the Einstein flow.

\quad In $2021$, S. Roy et al. \cite{new} introduced the concept of conformal Einstein soliton as a generalization of the classical Einstein soliton (\ref{1.1}). According to \cite{new}, a Riemannian metric $g$ defined on a smooth manifold $\mathcal{M}$ of dimension $(2m+1)$ is called a conformal Einstein soliton if there exists a real constant $\Omega$ and a smooth vector field $\mathcal{V}$ on $\mathcal{M}$ such that the following equation holds,

\begin{equation}\label{1.2} Ric+\frac{1}{2}\pounds_{\mathcal{V}}g+\{\Omega-\frac{r}{2}+(\frac{p}{2}+\frac{1}{2m+1})\}g=0,\end{equation}\\ where $p$ is a scalar non-dynamical field. Here, $p$ is called conformal pressure.\\

\quad In $2002$, T. Hamada \cite{11} defined the $\ast$-Ricci tensor of type $(0,2)$ on real hypersurfaces of complex space forms by 
\begin{align*} Ric^{\ast}(\mathcal{X},\mathcal{Y})=g(Q^{\ast}\mathcal{X},\mathcal{Y})=\frac{1}{2}[trace(\varphi\otimes\mathcal{R}(\mathcal{X},\varphi\mathcal{Y}))]\end{align*}\\for any smooth vector field $\mathcal{X}, \mathcal{Y}$ on $\mathcal{M}$, where $\varphi$ is a $(1,1)$-tensor field and $\mathcal{R}$ is the Riemannian curvature tensor of $\mathcal{M}$, and $Q^{\ast}$ is the $\ast$-Ricci curvature operator. The $\ast$-scalar curvature of $\mathcal{M}$ is denoted by $r^{\ast}$ and is defined by $r^{\ast}=trace(Q^{\ast})$. A Riemannian manifold $(\mathcal{M},g)$ is called $\ast$-Ricci flat if $Ric^{\ast}$ vanishes identically. Over the years, several notion related to the $\ast$-Ricci tensor were initiated. In \cite{12}, the authors initiated the notion of $\ast$-Ricci soliton and widely studied by many authors \cite{8,9,16} and others. The $\ast$-Ricci soliton equation was given by,

\begin{equation}\label{1.3}
Ric^{\ast}+\frac{1}{2}\pounds_{\mathcal{V}}g+\Omega{g}=0,\end{equation}\\ where $\Omega$ is a real constant and the tensor $Ric^{\ast}$ of type $(0,2)$ is symmetric in $\mathcal{X}$ and $\mathcal{Y}$ on $\mathcal{M}$. Moreover, if the potential vector field $\mathcal{V}$ on $\mathcal{M}$ is of gradient type, i.e., $\mathcal{V}=grad(f)$, for some real valued smooth function $f$ defined on $\mathcal{M}$, then the soliton (\ref{1.3}) is called a $\ast$-conformal gradient Ricci soliton. In this paper, we introduce the notion of $\ast$-conformal Einstein soliton defined as:

\begin{defi}\label{defn1.1} A Riemannian or pseudo-Riemannian manifold $(\mathcal{M},g)$ of dimension $(2m+1)$ is said to admit $\ast$-conformal Einstein ($\ast$-CE) soliton of the type $(g,\mathcal{V},\Omega)$ if 

\begin{equation}\label{1.4} Ric^{\ast}+\frac{1}{2}\pounds_{\mathcal{V}}g+\{\Omega-\frac{r^{\ast}}{2}+(\frac{p}{2}+\frac{1}{2m+1})\}g=0,\end{equation}\\ where $\Omega$ is a real constant, provided the tensor $Ric^{\ast}$ of type $(0,2)$ is symmetric in $\mathcal{X}$ and $\mathcal{Y}$ on $\mathcal{M}$.
 Moreover, if $\mathcal{V}=grad(f)=\nabla{f}$, for some real valued smooth function $f$ defined on $\mathcal{M}$, where $\nabla$ is the Riemannian connection defined on $\mathcal{M}$, then the soliton (\ref{1.4}) is called $\ast$-conformal gradient Einstein ($\ast$-CGE) soliton of the type $(g,\nabla{f},\Omega)$. In this occasion, the ($\ast$-CGE) soliton of the type $(g,\nabla{f},\Omega)$ is given by 

\begin{equation}\label{1.5} Ric^{\ast}+\nabla\nabla{f}+\{\Omega-\frac{r^{\ast}}{2}+(\frac{p}{2}+\frac{1}{2m+1})\}g=0.\end{equation}

The $\ast$-conformal Einstein ($\ast$-CE) soliton or $\ast$-conformal gradient Einstein ($\ast$-CGE) soliton is said to be shrinking if $\Omega<0$, steady if $\Omega=0$ and expanding if $\Omega>0$.
\end{defi}

\quad In the present paper we study $\ast$-conformal Einstein ($\ast$-CE) soliton and $\ast$-conformal gradient Einstein ($\ast$-CGE) soliton on $\mathcal{N}(k)$-contact metric manifolds. The paper is organized in the following way: After introduction, we have discussed some essential definitions and curvature formulas related to $\mathcal{N}(k)$-contact metric manifolds, which are contained in section \ref{sec2}. In section \ref{sec3}, we have studied $\ast$-conformal Einstein ($\ast$-CE) soliton on $\mathcal{N}(k)$-contact metric manifold and characterized the behaviour of the manifold, potential vector field of the soliton. Section \ref{sec4} deals with $\mathcal{N}(k)$-contact metric manifold whose metric $g$ satisfies $\ast$-conformal gradient Einstein ($\ast$-CGE) soliton and obtain some impotent results. Finally, we present an example to justify our results that we obtain in this work.

\section{Preliminaries}\label{sec2} 

\quad A smooth manifold $\mathcal{M}$ of dimension $(2m+1)$ is said to have an almost contact structure or $(\varphi,\zeta,\eta)$ structure \cite{3} if $\mathcal{M}$ permits a $(1,1)$-tensor field $\varphi$, a smooth vector field $\zeta$ (called the characteristic vector field), and an $1$-form $\eta$ satisfying
\begin{equation}\label{2.1}\varphi^2(\mathcal{X})=-\mathcal{X}+\eta(\mathcal{X})\zeta,\hspace{1.0cm}\eta(\zeta)=1\end{equation}\\for any smooth vector field $\mathcal{X}$ on $\mathcal{M}$. Then from (\ref{2.1}) it follows that 

\begin{equation}\label{2.2}\varphi(\zeta)=0,\hspace{0.5cm}\eta(\varphi\mathcal{X})=0.\end{equation}

\quad If $\mathcal{M}$ with an almost contact structure $(\varphi,\zeta,\eta)$ admits a Riemannian metric $g$ such that

\begin{equation}\label{2.3}g(\varphi\mathcal{X},\varphi\mathcal{Y})=g(\mathcal{X},\mathcal{Y})-\eta(\mathcal{X})\eta(\mathcal{Y}),\hspace{0.5cm} g(\mathcal{X},\zeta)=\eta(\mathcal{X})\end{equation}\\for any smooth vector fields $\mathcal{X}, \mathcal{Y}$ on $\mathcal{M}$, then $(\varphi,\zeta,\eta, g)$ is called an almost contact metric structure. A manifold having almost contact metric structure is called an almost contact metric manifold. From (\ref{2.3}), it can be easily deduced that  
\begin{equation}\label{2.4}g(\varphi\mathcal{X},\mathcal{Y})=-g(\mathcal{X},\varphi\mathcal{Y}).\end{equation}

\quad An almost contact metric structure $(\varphi,\zeta,\eta, g)$ becomes a contact metric structure if 

\begin{equation}\label{2.5}g(\varphi\mathcal{X},\mathcal{Y})=d\eta(\mathcal{X},\mathcal{Y})\end{equation}

for any smooth vector fields $\mathcal{X}, \mathcal{Y}$ on $\mathcal{M}$, where $d$ denotes the exterior differentiation. 

\quad On a contact metric manifold, the $(1,1)$-tensor field $h$ is defined as $h=\frac{1}{2}\pounds_{\zeta}\varphi$, where $\pounds_{\zeta}\varphi$ denotes the Lie derivative of $\varphi$ by $\zeta$. The $(1,1)$-tensor field $h$ is symmetric and satisfies
\begin{equation}\label{2.6}h\varphi+\varphi{h}=0,\hspace{0.5cm} trace(h)=trace(\varphi{h})=0,\hspace{0.5cm}h\zeta=0.\end{equation} Also, we have 
\begin{equation}\label{2.7} \nabla_{\mathcal{X}}\zeta=-\varphi{\mathcal{X}}-\varphi{h}\mathcal{X}\end{equation}\\for any smooth vector field $\mathcal{X}$ on $\mathcal{M}$, where  $\nabla$ is the Levi-Civita connection of $g$ on $\mathcal{M}$.\\

\quad In \cite{b}, Blair et al. defined the notion of ${(k,\mu)}$-nullity distribution on contact metric manifold as follows:

\begin{equation}\label{a}\begin{array}{cc}
\mathcal{N}(k,\mu):q\rightarrow{\mathcal{N}_q}{(k,\mu)}=\{\mathcal{W}\in{\mathcal{T}_q}(\mathcal{M}): \mathcal{R}(\mathcal{X},\mathcal{Y})\mathcal{W}=(kI+\mu{h})[g(\mathcal{Y},\mathcal{W})\mathcal{X}\\\\-g(\mathcal{X},\mathcal{W})\mathcal{Y}]\},\end{array}\end{equation}

 where $({k},\mu)\in\mathbb{R}^2$, ${I}$ is an identity mapping and ${\mathcal{T}_q(\mathcal{M}})$ denotes the tangent vector space of $\mathcal{M}$ at point $q$. If the characteristic vector field $\xi$ belongs to ${(k,\mu)}$-nullity distribution $\mathcal{N}(k,\mu)$, then we call a contact metric manifold as ${(k,\mu)}$-contact metric manifold. Also, the contact metric manifold $\mathcal{M}$ is called $\mathcal{N}(k)$-contact metric manifold \cite{18} if it satisfies $(\ref{a})$ with $\mu=0$ and $\xi$ belongs to ${k}$-nullity distribution $\mathcal{N}(k)$. In the case of $\mathcal{N}(k)$-contact metric manifold, the ${k}$-nullity distribution is given by

$$\mathcal{N}(k):q\rightarrow{\mathcal{N}_q}{(k)}=\{\mathcal{W}\in{\mathcal{T}_q}(\mathcal{M}): \mathcal{R}(\mathcal{X},\mathcal{Y})\mathcal{W}=k[g(\mathcal{Y},\mathcal{W})\mathcal{X}-g(\mathcal{X},\mathcal{W})\mathcal{Y}]\}.$$

\quad For a $(2m+1)$-dimensional $\mathcal{N}(k)$-contact metric manifold the following relations are satisfied \cite{1,3}:
\begin{equation}\label{2.8}h^2=(k-1)\varphi^2,\end{equation}

\begin{equation}\label{2.9}\mathcal{R}(\mathcal{X},\mathcal{Y})\zeta=k\{\eta(\mathcal{Y})\mathcal{X}-\eta(\mathcal{X})\mathcal{Y}\},\end{equation}

\begin{equation}\label{2.10}\mathcal{R}(\zeta,\mathcal{X})\mathcal{Y}=k\{g(\mathcal{X},\mathcal{Y})\zeta-\eta(\mathcal{Y})\mathcal{X}\},\end{equation}

\begin{equation}\label{2.11}(\nabla_{\mathcal{X}}\eta)\mathcal{Y}=g(\mathcal{X}+h\mathcal{X},\varphi\mathcal{Y}),\end{equation}

\begin{equation}\label{2.12}(\nabla_{\mathcal{X}}\varphi)\mathcal{Y}=g(\mathcal{X}+h\mathcal{X},\mathcal{Y})\zeta-\eta(\mathcal{Y})(\mathcal{X}+h\mathcal{X}),\end{equation}

\begin{equation}\label{2.13}\begin{array}{cc}(\nabla_{\mathcal{X}}\varphi{h})\mathcal{Y}=[g(\mathcal{X},h\mathcal{Y})-(k-1)g(\mathcal{X},\mathcal{Y}-\eta(\mathcal{Y})\zeta)]\zeta\\\\ \hspace{1.5cm}+\eta(\mathcal{Y})[h\mathcal{X}-(k-1)(\mathcal{X}-\eta(\mathcal{X})\zeta)],\end{array}\end{equation}

\begin{equation}\label{2.14}\begin{array}{cc}{Ric}(\mathcal{X},\mathcal{Y})=2(m-1)\{g(\mathcal{X},\mathcal{Y})+g(h\mathcal{X},\mathcal{Y})\}+{2\{mk-m+1\}\eta(\mathcal{X})\eta(\mathcal{Y})},\end{array}\end{equation}

\begin{equation}\label{2.15}{Ric}(\mathcal{X},\zeta)=2mk\eta({\mathcal{X}})\end{equation}\\for any smooth vector fields $\mathcal{X}, \mathcal{Y}$ on $\mathcal{M}$, where $\mathcal{R}$, $Ric$ are the Riemann curvature tensor, and symmetric Ricci tensor of type $(0,2)$ of $\mathcal{M}$, respectively, and $\mathsf{{m}\geq{1}}$. For more details about the $\mathcal{N}(k)$-contact metric manifolds, we cite \cite{10,13,14} and the references therein.

\begin{defi}\label{defn2.1}
On a $(2m+1)$-dimensional Riemannian manifold $(\mathcal{M},g)$, a smooth vector field $\mathcal{V}$ is said to be conformal \cite{20,21} if

\begin{equation}\label{2.16}\pounds_{\mathcal{V}}g=2\rho{g}\end{equation}\\holds for some real valued  smooth function $\rho$ difined on $\mathcal{M}$. The function $\rho$ is also known as conformal coefficient. In particular, the conformal vector field with a vanishing conformal coefficient reduces to the Killing vector field.
\end{defi} 

\begin{defi}\label{defn2.2}
A smooth vector field $\mathcal{V}$ on a contact metric manifold $\mathcal{M}$ is said to be an infinitesimal contact transformation \cite{19} if it preserves the contact form $\eta$, i.e., there exists a real valued  smooth function $\psi$ on $\mathcal{M}$ that satisfies  

\begin{equation}\label{2.17}\pounds_{\mathcal{V}}\eta=\psi{g},\end{equation}\\ where $\pounds_{\mathcal{V}}\eta$ denotes the Lie derivative of $\eta$ by $\mathcal{V}$.
In particular, if the smooth function $\psi$ vanishes identically in (\ref{2.17}), then $\mathcal{V}$ is said to be a strict infinitesimal contact transformation.
\end{defi}

\section{{$\mathcal{N}(k)$-contact metric manifold admitting $\ast$-conformal Einstein soliton}}\label{sec3}

\quad This section is devoted to the study of $\mathcal{N}(k)$-contact metric manifold admitting a ($\ast$-CE) soliton of the type $(g,\mathcal{V},\Omega)$. To produce our prime theorems, we need the following Lemmas:
\begin{lem}\label{lem3.1}
(\cite{2}) A $(2m+1)$-dimensional contact metric manifold $\mathcal{M}$ satisfying $\mathcal{R}(\mathcal{X},\mathcal{Y})\zeta=0$ for any smooth vector fields $\mathcal{X},\mathcal{Y}$ on $\mathcal{M}$ is locally isometric to the Riemannian product of a flat manifold of dimension $(m+1)$ and a manifold of dimension $m$ whose positive curvature is equal to $4$, that is, $\mathbb{E}^{m+1}(0)\times{\mathcal{S}^m(4)}$ for $m>1$ and flat for $m=1$. 
\end{lem}

\begin{lem}\label{lem3.2}
(\cite{9}) On a $(2m+1)$-dimensional $\mathcal{N}(k)$-contact metric manifold $\mathcal{M}$, the $\ast$-Ricci tensor $Ric^{\ast}$ of type $(0,2)$ is given by 

\begin{equation}\label{3.1} Ric^{\ast}(\mathcal{X},\mathcal{Y})=k\{\eta(\mathcal{X})\eta(\mathcal{Y})-g(\mathcal{X},\mathcal{Y})\} \end{equation}\\for any smooth vector fields $\mathcal{X}, \mathcal{Y}$ on $\mathcal{M}$.
\end{lem}

\quad Note that, $\ast$-Ricci tensor ${Ric^{\ast}}$ is not symmetric in general. But from (\ref{3.1}) it is observe that in a $\mathcal{N}(k)$-contact metric manifold, ${Ric^{\ast}}$ is symmetric and hence, the definition \ref{defn1.1} is well defined on $\mathcal{N}(k)$-contact metric manifold of dimension greater than or equal to three. 

\quad Taking $\mathcal{X}=\mathcal{Y}=e_i$ in (\ref{3.1}), where $\{e_i\}_{i=1}^{2m+1}$ is an orthonormal basis of the tangent vector space at each point of $\mathcal{M}$, and summing over $i\in\{1,2,3,...,2m+1\}$ we obtain 

\begin{equation}\label{3.2} r^{\ast}=-2mk.\end{equation}

\begin{lem}\label{lem3.3}
If a $(2m+1)$-dimensional $\mathcal{N}(k)$-contact metric manifold $\mathcal{M}$ admits a ($\ast$-CE) soliton of the type $(g,\mathcal{V},\Omega)$, then the $\ast$-Ricci tensor ${Ric^{\ast}}$ satisfies 

\begin{equation}\begin{array}{cc}\label{3.3}(\nabla_{\mathcal{W}}Ric^{\ast})(\mathcal{X},\mathcal{Y})-(\nabla_{\mathcal{X}}Ric^{\ast})(\mathcal{Y},\mathcal{W})-(\nabla_{\mathcal{Y}}Ric^{\ast})(\mathcal{X},\mathcal{W}) =\\\\2k\{g(\mathcal{W},\varphi\mathcal{Y})\eta(\mathcal{X})+g(\mathcal{W},\varphi\mathcal{X})\eta(\mathcal{Y})-g(h\mathcal{X},\varphi\mathcal{Y})\eta(\mathcal{W})\} 
\end{array}\end{equation}\\for any smooth vector fields $\mathcal{X,Y,W}$ on $\mathcal{M}$.
\end{lem}

\begin{proof} Recalling the equation (\ref{3.1}) and covariantly differentiating it along an arbitrary smooth vector field $\mathcal{W}$ on $\mathcal{M}$ yields 

\begin{equation}\label{3.4}\nabla_{\mathcal{W}}Ric^{\ast}(\mathcal{X},\mathcal{Y})=k\{(\nabla_{\mathcal{W}}\eta(\mathcal{X}))\eta(\mathcal{Y})+(\nabla_{\mathcal{W}}\eta(\mathcal{Y}))\eta(\mathcal{X})-\nabla_{\mathcal{W}}g(\mathcal{X},\mathcal{Y})\}. 
\end{equation}\\We know that, $$(\nabla_{\mathcal{W}}Ric^{\ast})(\mathcal{X},\mathcal{Y})=\nabla_{\mathcal{W}}Ric^{\ast}(\mathcal{X},\mathcal{Y})-Ric^{\ast}(\nabla_{\mathcal{W}}\mathcal{X},\mathcal{Y})-Ric^{\ast}(\mathcal{X},\nabla_{\mathcal{W}}\mathcal{Y}).$$ \\Using (\ref{3.1}) and (\ref{3.4}) in the forgoing equation infers that 

\begin{equation}\begin{array}{rcl}\label{3.5}(\nabla_{\mathcal{W}}Ric^{\ast})(\mathcal{X},\mathcal{Y})=k\{((\nabla_{\mathcal{W}}\eta)(\mathcal{X}))\eta(\mathcal{Y})+((\nabla_{\mathcal{W}}\eta)(\mathcal{Y}))\eta(\mathcal{X})\} \\\\= k\{g(\mathcal{W}+h\mathcal{W},\varphi\mathcal{X})\eta(\mathcal{Y})+g(\mathcal{W}+h\mathcal{W},\varphi\mathcal{Y})\eta(\mathcal{X})\},
\end{array}\end{equation}\\ where we have applied (\ref{2.12}). Similarly from (\ref{3.5}), we obtained 

\begin{equation}\begin{array}{rcl}\label{3.6}
(\nabla_{\mathcal{X}}Ric^{\ast})(\mathcal{Y},\mathcal{W})=k\{g(\mathcal{X}+h\mathcal{X},\varphi\mathcal{Y})\eta(\mathcal{W})+g(\mathcal{X}+h\mathcal{X},\varphi\mathcal{W})\eta(\mathcal{Y})\},
\end{array}\end{equation}\\and \begin{equation}\begin{array}{rcl}\label{3.7}
(\nabla_{\mathcal{Y}}Ric^{\ast})(\mathcal{X},\mathcal{W})=k\{g(\mathcal{Y}+h\mathcal{Y},\varphi\mathcal{X})\eta(\mathcal{W})+g(\mathcal{Y}+h\mathcal{Y},\varphi\mathcal{W})\eta(\mathcal{X})\}.
\end{array}\end{equation}\\Finally with the help of (\ref{3.5})-(\ref{3.7}) and taking reference of (\ref{2.4}) and (\ref{2.6}), we complete the proof.
\end{proof}

\begin{thm}\label{thm3.3}
If a $(2m+1)$-dimensional $\mathcal{N}(k)$-contact metric manifold $\mathcal{M}$ admits a ($\ast$-CE) soliton of the type $(g,\mathcal{V},\Omega)$, then we have following:\\${\bf(i)}$ If $m>1$, then $\mathcal{M}$ is locally isometric to $\mathbb{E}^{m+1}(0)\times{\mathcal{S}^m(4)}$ and $\mathcal{M}$ is flat for $m=1$.\\${\bf(ii)}$ The manifold $\mathcal{M}$ is $\ast$-Ricci flat.\\${\bf(iii)}$ The potential vector field $\mathcal{V}$ is conformal,\\provided that $\Omega+mk+(\frac{p}{2}+\frac{1}{2m+1})\neq{0}$.
\end{thm}
\begin{proof} 

Let us assume that $(g,\mathcal{V},\Omega)$ be a ($\ast$-CE) soliton on a $(2m+1)$-dimensional $\mathcal{N}(k)$-contact metric manifold $\mathcal{M}$. Then, in view of (\ref{3.2}), the soliton equation (\ref{1.4}) becomes 

\begin{equation}\label{3.8} Ric^{\ast}(\mathcal{X,Y})+\frac{1}{2}(\pounds_{\mathcal{V}}g)(\mathcal{X,Y})+\{\Omega+mk+(\frac{p}{2}+\frac{1}{2m+1})\}g(\mathcal{X,Y})=0,\end{equation}\\for any smooth vector fields $\mathcal{X,Y}$ on $\mathcal{M}$.\\\\The covariant derivative of (\ref{3.8}) along an arbitrary smooth vector field $\mathcal{W}$ on $\mathcal{M}$ gives

\begin{equation}\label{3.9}
(\nabla_{\mathcal{W}}\pounds_{\mathcal{V}}g)(\mathcal{X},\mathcal{Y})=-2(\nabla_{\mathcal{W}}Ric^{\ast})(\mathcal{X},\mathcal{Y}).    
\end{equation}\\We have the following well-known commutation formula (see \cite{20}):

\begin{align*}\begin{array}{cc}(\pounds_{\mathcal{V}}\nabla_{\mathcal{X}}g-\nabla_{\mathcal{X}}\pounds_{\mathcal{V}}g-\nabla_{[\mathcal{V},\mathcal{X}]}g)(\mathcal{Y},\mathcal{W})=-g((\pounds_{\mathcal{V}}\nabla)(\mathcal{X},\mathcal{Y}),\mathcal{W})\\\\-g((\pounds_{\mathcal{V}}\nabla)(\mathcal{X},\mathcal{W}),\mathcal{Y}).
\end{array}\end{align*}\\Since $g$ is parallel, i.e. $\nabla{g}=0$, the above commutation formula becomes 

\begin{equation}\label{3.10}
(\nabla_{\mathcal{X}}\pounds_{\mathcal{V}}g)(\mathcal{Y},\mathcal{W})=g((\pounds_{\mathcal{V}}\nabla)(\mathcal{X},\mathcal{Y}),\mathcal{W})+g((\pounds_{\mathcal{V}}\nabla)(\mathcal{X},\mathcal{W}),\mathcal{Y}).   
\end{equation}\\Since $\pounds_{\mathcal{V}}\nabla$ is a symmetric operator, i.e. $(\pounds_{\mathcal{V}}\nabla)(\mathcal{X},\mathcal{Y})=(\pounds_{\mathcal{V}}\nabla)(\mathcal{Y},\mathcal{X})$, then it follows from (\ref{3.10}) that

\begin{equation}\begin{array}{cc}\label{3.11}
g((\pounds_{\mathcal{V}}\nabla)(\mathcal{X},\mathcal{Y}),\mathcal{W})=\frac{1}{2}(\nabla_{\mathcal{X}}\pounds_{\mathcal{V}}g)(\mathcal{Y},\mathcal{W})+\frac{1}{2}(\nabla_{\mathcal{Y}}\pounds_{\mathcal{V}}g)(\mathcal{X},\mathcal{W})-\frac{1}{2}(\nabla_{\mathcal{W}}\pounds_{\mathcal{V}}g)(\mathcal{X},\mathcal{Y}).
\end{array}\end{equation} In view of (\ref{3.9}), the forgoing equation becomes   

\begin{equation}\begin{array}{cc}\label{3.12}
g((\pounds_{\mathcal{V}}\nabla)(\mathcal{X},\mathcal{Y}),\mathcal{W})=(\nabla_{\mathcal{W}}Ric^{\ast})(\mathcal{X},\mathcal{Y})-(\nabla_{\mathcal{X}}Ric^{\ast})(\mathcal{Y},\mathcal{W})-(\nabla_{\mathcal{Y}}Ric^{\ast})(\mathcal{X},\mathcal{W}).
\end{array}\end{equation} Using the lemma \ref{lem3.3} in equation (\ref{3.12}) we get

\begin{equation}\begin{array}{cc}\label{3.13}
g((\pounds_{\mathcal{V}}\nabla)(\mathcal{X},\mathcal{Y}),\mathcal{W})=2k\{g(\mathcal{W},\varphi\mathcal{Y})\eta(\mathcal{X})+g(\mathcal{W},\varphi\mathcal{X})\eta(\mathcal{Y})-g(h\mathcal{X},\varphi\mathcal{Y})\eta(\mathcal{W})\}\end{array}\end{equation} which gives
\begin{equation}\label{3.14}
(\pounds_{\mathcal{V}}\nabla)(\mathcal{X},\mathcal{Y})=2k\{\eta(\mathcal{X})\varphi\mathcal{Y}+\eta(\mathcal{Y})\varphi\mathcal{X}-g(h\mathcal{X},\varphi\mathcal{Y})\zeta\}.\end{equation}\\Replacing $\mathcal{Y}$ by $\zeta$ in (\ref{3.14}) and using (\ref{2.2}) we have

\begin{equation}\label{3.15} 
(\pounds_{\mathcal{V}}\nabla)(\mathcal{X},\zeta)=2k\varphi\mathcal{X}.\end{equation}\\Taking covariant derivative of (\ref{3.15}) along an arbitrary smooth vector field $\mathcal{Y}$ on $\mathcal{M}$ yields

\begin{equation}\label{3.16} 
\nabla_{\mathcal{Y}}(\pounds_{\mathcal{V}}\nabla)(\mathcal{X},\zeta)=2k\nabla_{\mathcal{Y}}(\varphi\mathcal{X}).\end{equation}\\Again from the definition of covariant derivative we have 

\begin{align*}(\nabla_{\mathcal{Y}}\pounds_{\mathcal{V}}\nabla)(\mathcal{X},\zeta)=\nabla_{\mathcal{Y}}(\pounds_{\mathcal{V}}\nabla)(\mathcal{X},\zeta)-(\pounds_{\mathcal{V}}\nabla)(\nabla_{\mathcal{Y}}\mathcal{X},\zeta)-(\pounds_{\mathcal{V}}\nabla)(\mathcal{X},\nabla_{\mathcal{Y}}\zeta).
\end{align*}\\Now using (\ref{3.14})-(\ref{3.16}) in above formula we obtain

\begin{equation}\begin{array}{cc}\label{3.17}(\nabla_{\mathcal{Y}}\pounds_{\mathcal{V}}\nabla)(\mathcal{X},\zeta)=2k\{(\nabla_{\mathcal{Y}}\varphi)\mathcal{X}-\eta(\mathcal{X})\varphi(\nabla_{\mathcal{Y}}\zeta)-\eta(\nabla_{\mathcal{Y}}\zeta)\varphi\mathcal{X}+g(h\mathcal{X},\varphi(\nabla_{\mathcal{Y}}\zeta))\zeta\} .\end{array}\end{equation}Further, with the help of (\ref{2.7}) and (\ref{2.12}), the forgoing equation becomes 

\begin{equation}\begin{array}{cc}\label{3.18}(\nabla_{\mathcal{Y}}\pounds_{\mathcal{V}}\nabla)(\mathcal{X},\zeta)=2k\{g(\mathcal{X},\mathcal{Y})\zeta-g(h\mathcal{X},h\mathcal{Y})\zeta-\eta(\mathcal{Y})(\mathcal{X}+h\mathcal{X})-\eta(\mathcal{X})(\mathcal{Y}+h\mathcal{Y})\\\\+\eta(\mathcal{X})\eta(\mathcal{Y})\zeta\}.
\end{array}\end{equation} From K. Yano \cite{20}, we know the following well-known  formula:

\begin{align*}(\pounds_{\mathcal{V}}\mathcal{R})(\mathcal{X},\mathcal{Y})\mathcal{W}=(\nabla_{\mathcal{X}}\pounds_{\mathcal{V}}\nabla)(\mathcal{Y},\mathcal{W})-(\nabla_{\mathcal{Y}}\pounds_{\mathcal{V}}\nabla)(\mathcal{X},\mathcal{W}).
\end{align*}\\Making use of (\ref{3.18}) in the above formula, we obtain

\begin{equation}\label{3.19}(\pounds_{\mathcal{V}}\mathcal{R})(\mathcal{X},\zeta)\zeta=(\nabla_{\mathcal{X}}\pounds_{\mathcal{V}}\nabla)(\zeta,\zeta)-(\nabla_{\zeta}\pounds_{\mathcal{V}}\nabla)(\mathcal{X},\zeta)=0.\end{equation}\\Next, recalling the equation (\ref{3.8}) and replacing $\zeta$ instead of $\mathcal{Y}$, we get 

$$(\pounds_{\mathcal{V}}g)(\mathcal{X},\zeta)+2\{\Omega+mk+(\frac{p}{2}+\frac{1}{2m+1})\}\eta(\mathcal{X})=0,$$ which gives 
\begin{equation}\label{3.20}
(\pounds_{\mathcal{V}}\eta)\mathcal{X}-g(\mathcal{X},\pounds_{\mathcal{V}}\zeta)+2\{\Omega+mk+(\frac{p}{2}+\frac{1}{2m+1})\}\eta(\mathcal{X})=0.\end{equation}\\After replacing $\zeta$ instead of $\mathcal{X}$ in forgoing equation, it can be easily deduced that

\begin{equation}\label{3.21}
\eta(\pounds_{\mathcal{V}}\zeta)=\{\Omega+mk+(\frac{p}{2}+\frac{1}{2m+1})\}.\end{equation}\\On the other hand, from (\ref{2.9}), we have  

\begin{equation}\label{3.22}
\mathcal{R}{(\mathcal{X},\zeta)}\zeta=k\{\mathcal{X}-\eta(\mathcal{X})\zeta\}.
\end{equation}\\Taking Lie derivative of (\ref{3.22}) along $\mathcal{V}$ and using (\ref{2.10}) and (\ref{3.20})-(\ref{3.22}) yields

\begin{equation}\label{3.23}(\pounds_{\mathcal{V}}\mathcal{R})(\mathcal{X},\zeta)\zeta=2k\{\Omega+mk+(\frac{p}{2}+\frac{1}{2m+1})\}(\eta(\mathcal{X})\zeta-\mathcal{X}).
\end{equation}\\Equating (\ref{3.19}) with (\ref{3.23}) we deduce

\begin{equation}\label{3.24}k\{\Omega+mk+(\frac{p}{2}+\frac{1}{2m+1})\}(\eta(\mathcal{X})\zeta-\mathcal{X})=0.
\end{equation}\\Executing the inner product of (\ref{3.24}) with $\mathcal{Y}$ and using the first equation of (\ref{2.3}) yields

\begin{equation}\label{3.25}
k\{\Omega+mk+(\frac{p}{2}+\frac{1}{2m+1})\}g(\varphi\mathcal{X},\varphi\mathcal{Y})=0.\end{equation}\\
Since the above holds for all smooth vector fields $\mathcal{X,Y}$ on the manifold, and by assumption, ${\Omega+mk+(\frac{p}{2}+\frac{1}{2m+1})}\neq{0}$, we immediately have $k=0$. Using this in equation (\ref{3.1}) infers that $Ric^{\ast}=0$ and hence $\mathcal{N}(k)$-contact metric manifold $\mathcal{M}$ is $\ast$-Ricci flat, and this proves ${\bf(ii)}$. Again from (\ref{2.9}) we get $\mathcal{R}(\mathcal{X},\mathcal{Y})\zeta=0$ for any smooth vector fields $\mathcal{X}, \mathcal{Y}$ on $\mathcal{M}$ and so it follows from Lemma \ref{lem3.1} that $\mathcal{M}$ is locally isometric to $\mathbb{E}^{m+1}(0)\times{\mathcal{S}^m(4)}$ for $m>1$ and flat for $m=1$. This proves ${\bf(i)}$ of Theorem \ref{thm3.3}.\\

Again using the fact that $k=0$ and $Ric^{\ast}=0$ in the soliton equation (\ref{3.8}) we obtain

\begin{equation}\label{3.26}
(\pounds_{\mathcal{V}}g)(\mathcal{X},\mathcal{Y})=-2\{\Omega+(\frac{p}{2}+\frac{1}{2m+1})\}{g(\mathcal{X},\mathcal{Y})},\end{equation} which can be written that 
$$(\pounds_{\mathcal{V}}g)(\mathcal{X},\mathcal{Y})=2\rho{g(\mathcal{X},\mathcal{Y})}$$ for any smooth vector fields $\mathcal{X}, \mathcal{Y}$ on $\mathcal{M}$, where $\rho=-\Omega-(\frac{p}{2}+\frac{1}{2m+1})$. Thus in the sense of definition \ref{defn2.1}, the potential vector field $\mathcal{V}$ is a conformal vector field with conformal
coefficient $-\Omega-(\frac{p}{2}+\frac{1}{2m+1})$. This proves ${\bf(iii)}$ and hence completes the proof.
\end{proof}

\begin{rem}\label{rem3.5}
If $k=0$ and ${\Omega+mk+(\frac{p}{2}+\frac{1}{2m+1})}=0$, then it can be easily deduced that the soliton constant $\Omega=-\frac{p}{2}-\frac{1}{2m+1}$. Therefore, the ($\ast$-CE) soliton is shrinking if $p>{-\frac{2}{2m+1}}$, steady if $p={-\frac{2}{2m+1}}$ or, expanding if $p<{-\frac{2}{2m+1}}$. Also, we obtain from $(\ref{3.26})$ that $(\pounds_{\mathcal{V}}g)(\mathcal{X},\mathcal{Y})=0$ for any smooth vector fields $\mathcal{X}, \mathcal{Y}$ on $\mathcal{M}$ and hence the potential vector field $\mathcal{V}$ is a Killing vector field. 
\end{rem}

\begin{rem}\label{rem3.6}
If $k\neq{0}$ and ${\Omega+mk+(\frac{p}{2}+\frac{1}{2m+1})}=0$, then from (\ref{3.20}) it can be easily deduced that $(\pounds_{\mathcal{V}}\eta)\mathcal{X}=g(\mathcal{X},\pounds_{\mathcal{V}}\zeta)$. Therefore the potential vector field $\mathcal{V}$ will be an infinitesimal contact transformation if $\pounds_{\mathcal{V}}\zeta=\psi\zeta$ for some real valued smooth function $\psi$ defined on $\mathcal{M}$. Also from (\ref{3.21}), we have $\eta(\pounds_{\mathcal{V}}\zeta)=0$. This eventually implies that $\pounds_{\mathcal{V}}\zeta\bot\zeta$. Thus $\pounds_{\mathcal{V}}\zeta\neq{\psi\zeta}$, unless $\psi\equiv{0}.$ 
Hence in the sense of definition (\ref{defn2.2}), $\mathcal{V}$ cannot be an infinitesimal contact transformation on $\mathcal{M}$ but it can be a strict infinitesimal contact transformation on $\mathcal{M}$ if $\pounds_{\mathcal{V}}\zeta=0$.
\end{rem}


\section{{$\mathcal{N}(k)$-contact metric manifold admitting $\ast$-conformal Einstein gradient soliton}}\label{sec4}

In this section, we now concentrate on $\mathcal{N}(k)$-contact metric manifold admitting a ($\ast$-CGE) soliton of the type $(g,\nabla{f},\Omega)$. But before proving our prime theorem in this direction, let us first prove the following Lemma:

\begin{lem}\label{lem4.1}
If a $(2m+1)$-dimensional $\mathcal{N}(k)$-contact metric manifold $\mathcal{M}$ admits a ($\ast$-CGE) soliton of the type $(g,\nabla{f},\Omega)$, then the curvature tensor $\mathcal{R}$ of $\mathcal{M}$ satisfies 

\begin{align*} 
\mathcal{R}(\mathcal{X},\mathcal{Y})\nabla{f}=k\{2g(\varphi\mathcal{X},\mathcal{Y})-\eta(\mathcal{X})(\varphi\mathcal{Y}+\varphi{h}\mathcal{Y})+\eta(\mathcal{Y})(\varphi\mathcal{X}+\varphi{h}\mathcal{X})\}
\end{align*}\\for all smooth vector fields $\mathcal{X},\mathcal{Y}$ on $\mathcal{M}$, where $\nabla{f}=grad(f)$ for some real valued smooth function $f$ defined on $\mathcal{M}$.
\end{lem}

\begin{proof} Let us assume that $(g,\nabla{f},\Omega)$ be a ($\ast$-CGE) soliton. Then the gradient soliton equation (\ref{1.5}) can be written as:

\begin{equation}\label{4.1}\nabla_{\mathcal{X}}\nabla{f}=-(\Omega+mk+\frac{p}{2}+\frac{1}{2m+1})\mathcal{X}-Q^{\ast}\mathcal{X}
\end{equation}\\for any smooth vector fields $\mathcal{X},\mathcal{Y}$ on $\mathcal{M}$.\\\\Executing the covariant derivative of (\ref{4.1}) along an arbitrary smooth vector field $\mathcal{Y}$ on $\mathcal{M}$ yields

\begin{equation}\label{4.2}\nabla_{\mathcal{Y}}\nabla_{\mathcal{X}}\nabla{f}=-(\Omega+mk+\frac{p}{2}+\frac{1}{2m+1})\nabla_{\mathcal{Y}}\mathcal{X}-\nabla_{\mathcal{Y}}Q^{\ast}\mathcal{X}. 
\end{equation}\\Exchanging the vector fields $\mathcal{X}$ and $\mathcal{Y}$ on $\mathcal{M}$ in the forgoing equation infers that

\begin{equation}\label{4.3}\nabla_{\mathcal{X}}\nabla_{\mathcal{Y}}\nabla{f}=-(\Omega+mk+\frac{p}{2}+\frac{1}{2m+1})\nabla_{\mathcal{X}}\mathcal{Y}-\nabla_{\mathcal{X}}Q^{\ast}\mathcal{Y}. 
\end{equation}\\Replacing $[\mathcal{X},\mathcal{Y}]$ instead of the vector field $\mathcal{X}$ on $\mathcal{M}$ in (\ref{4.1}) we get

\begin{equation}\label{4.4}\nabla_{[\mathcal{X},\mathcal{Y}]}\nabla{f}=(\Omega+mk+\frac{p}{2}+\frac{1}{2m+1})(\nabla_{\mathcal{Y}}\mathcal{X}-\nabla_{\mathcal{X}}\mathcal{Y})-Q^{\ast}(\nabla_{\mathcal{X}}\mathcal{Y}-\nabla_{\mathcal{Y}}\mathcal{X}).
\end{equation}\\Utilizing (\ref{4.2})-(\ref{4.4}) in the well known curvature formula  
$\mathcal{R}(\mathcal{X},\mathcal{Y})=[\nabla_{\mathcal{X}},\nabla_{\mathcal{Y}}]-\nabla_{[\mathcal{X},\mathcal{Y}]},$ where the square brackets $[\hspace{0.05cm} ]$ stands for Lie brackets, we obtain\begin{equation}\label{4.5}\mathcal{R}(\mathcal{X},\mathcal{Y})\nabla{f}= (\nabla_{\mathcal{Y}}Q^{\ast})\mathcal{X}-(\nabla_{\mathcal{X}}Q^{\ast})\mathcal{Y}.\end{equation}\\Also from the identity (\ref{3.1}) it can be written that 

\begin{equation}\label{4.6}
Q^{\ast}\mathcal{X}=k\eta(\mathcal{X})\zeta-k\mathcal{X}.
\end{equation}\\Taking covariant differentiation of (\ref{4.6}) along $\mathcal{Y}$ on $\mathcal{M}$ and using (\ref{2.7}) we obtain

\begin{equation}\label{4.7}\nabla_{\mathcal{Y}}Q^{\ast}\mathcal{X}=k\nabla_{\mathcal{Y}}\eta(\mathcal{X})\zeta+k\eta(\mathcal{X})\nabla_{\mathcal{Y}}\zeta-k\nabla_{\mathcal{Y}}\mathcal{X}.
\end{equation}\\Thus, we deduce from (\ref{4.6}) and (\ref{4.7}) that

\begin{equation}\label{4.8}
(\nabla_{\mathcal{Y}}Q^{\ast})\mathcal{X}=kg(\mathcal{Y}+h\mathcal{Y},\varphi\mathcal{X})-k\eta(\mathcal{X})(\varphi\mathcal{Y}+\varphi{h}\mathcal{Y}).
\end{equation}\\ Again exchanging the vector fields $\mathcal{X}$ and $\mathcal{Y}$ on $\mathcal{M}$ in (\ref{4.8}) yields

\begin{equation}\label{4.9}
(\nabla_{\mathcal{X}}Q^{\ast})\mathcal{Y}=kg(\mathcal{X}+h\mathcal{X},\varphi\mathcal{Y})-k\eta(\mathcal{Y})(\varphi\mathcal{X}+\varphi{h}\mathcal{X}).
\end{equation}\\By virtue of (\ref{4.8}) and (\ref{4.9}) from (4.5) completes the proof.
\end{proof}

Making use of the above lemma \ref{lem4.1} we can prove the following theorem:

\begin{thm}\label{thm4.2}
If a $(2m+1)$-dimensional $\mathcal{N}(k)$-contact metric manifold $\mathcal{M}$ admits a $\ast$-conformal Einstein gradient soliton $(g,\nabla{f},\Omega)$, then we have following:\\${\bf(i)}$ If $m>1$, then $\mathcal{M}$ is locally isometric to $\mathbb{E}^{m+1}(0)\times{\mathcal{S}^m(4)}$ and $\mathcal{M}$ is flat for $m=1$.\\${\bf(ii)}$ The manifold $\mathcal{M}$ is $\ast$-Ricci flat.\\${\bf(iii)}$ The function $f$ is either harmonic or satisfy a physical Poisson equation. 
\end{thm}

\begin{proof} Let $(g,\nabla{f},\Omega)$ be a ($\ast$-CGE) soliton on $\mathcal{N}(k)$-contact metric manifold $\mathcal{M}$ of dimensional  $(2m+1)$. Then recalling the lemma \ref{lem4.1} and then replacing $\mathcal{X}$ by $\zeta$ we obtain 

\begin{equation}\label{4.12}\mathcal{R}(\zeta,\mathcal{Y})\nabla{f}=-k(\varphi\mathcal{Y}+\varphi{h}\mathcal{Y}),\end{equation}\\for any smooth vector fields $\mathcal{Y}$ on $\mathcal{M}$. Taking inner product of (\ref{4.12}) with arbitrary smooth vector field $\mathcal{W}$ on $\mathcal{M}$ we get 

\begin{equation}\label{4.13}g(\mathcal{R}(\zeta,\mathcal{Y})\nabla{f},\mathcal{W})=-k\{g(\varphi\mathcal{Y},\mathcal{W})+g(\varphi{h}\mathcal{Y},\mathcal{W})\}.
\end{equation}\\On the other hand, in view of (\ref{2.10}) and $g(\mathcal{R}(\zeta,\mathcal{Y})\nabla{f},\mathcal{W})+g(\mathcal{R}(\zeta,\mathcal{Y})\mathcal{W},\nabla{f})=0$, we deduce

\begin{equation}\label{4.14}g(\mathcal{R}(\zeta,\mathcal{Y})\nabla{f},\mathcal{W})=-k\{g(\mathcal{Y},\mathcal{W})\zeta(f)-\eta(\mathcal{W})\mathcal{Y}(f)\}.
\end{equation}\\Comparing the identities (\ref{4.13}) and (\ref{4.14}), we have

\begin{equation}\label{4.15}k\{g(\varphi\mathcal{Y},\mathcal{W})+g(\varphi{h}\mathcal{Y},\mathcal{W})-g(\mathcal{Y},\mathcal{W})\zeta(f)+\eta(\mathcal{W})\mathcal{Y}(f)\}=0.
\end{equation}\\Plugging $\mathcal{W}=\zeta$, the foregoing equation infers that

\begin{equation}\label{4.16}k\{\mathcal{Y}(f)-\eta(\mathcal{Y})\zeta(f)\}=0.
\end{equation}\\Since (\ref{4.16}) may be exhibited as $k\{g(\mathcal{Y},\nabla{f})-g(\mathcal{Y},(\zeta{f})\zeta)\}=0$, we get

\begin{equation}\label{4.17}k\{\nabla{f}-(\zeta{f})\zeta\}=0.\end{equation}\\Therefore, either $k=0$ or $\nabla{f}=(\zeta{f})\zeta$.\\

\quad {\bf Case ${\bf(i)}$}: For $k=0$, equation (\ref{3.1}) gives $Ric^{\ast}=0$, and hence the manifold $\mathcal{M}$ is $\ast$-Ricci flat. Also from (\ref{2.9}) we get $\mathcal{R}(\mathcal{X},\mathcal{Y})\zeta=0$ for any smooth vector fields $\mathcal{X}, \mathcal{Y}$ on $\mathcal{M}$. Therefore, from lemma \ref{lem3.1}, we conclude that $\mathcal{M}$ is locally isometric to $\mathbb{E}^{m+1}(0)\times{\mathcal{S}^m(4)}$ for $m>1$ and flat for $m=1$. On taking $Ric^{\ast}=r^{\ast}=0$ in equation (\ref{1.5}) and then tracing it we obtain $\Delta{f}=-(\Omega+\frac{p}{2}+\frac{1}{2m+1})({2m+1})$, $\Delta$ being the Laplace operator. This reflects that the function $f$ satisfies a Poisson equation.\\

\quad {\bf Case ${\bf(ii)}$}: For $k\neq{0}$, from (\ref{4.17}) we have $\nabla{f}=(\zeta{f})\zeta$. Now covariantly differentiating it along an arbitrary smooth vector field $\mathcal{X}$ on $\mathcal{M}$ and then using (\ref{2.7}) yields

\begin{equation}\label{4.18}\nabla_{\mathcal{X}}\nabla{f}=(\mathcal{X}(\zeta{f}))\zeta-(\zeta{f})(\varphi\mathcal{X}+\varphi{h}\mathcal{X}).\end{equation}\\Equating (\ref{4.18}) with (\ref{4.1}) yields

\begin{equation}\label{4.19}Q^{\ast}\mathcal{X}=-(\Omega+mk+\frac{p}{2}+\frac{1}{2m+1})\mathcal{X}-(\mathcal{X}(\zeta{f}))\zeta+(\zeta{f})(\varphi\mathcal{X}+\varphi{h}\mathcal{X}), 
\end{equation}\\Comparing the coefficients of the smooth vector fields $\varphi\mathcal{X},\hspace{0.1cm} \zeta$ and $\mathcal{X}$ from (\ref{4.19}) and (\ref{3.1}), we achieve the followings, respectively:
\begin{equation}\label{4.20}(\zeta{f})=0.\end{equation}
\begin{equation}\label{4.21}\mathcal{X}(\zeta{f})=-k\eta(\mathcal{X}).\end{equation}
\begin{equation}\label{4.22}\Omega+mk+\frac{p}{2}+\frac{1}{2m+1}=k.\end{equation}\\Making use of (\ref{4.20}) in (\ref{4.21}) we get $k=0$. This fact together with (\ref{4.22}) leads to $\Omega=-\frac{p}{2}-\frac{1}{2m+1}$. Using $k=0$ and by the same procedure as in {\bf case ${\bf(i)}$}, the manifold $\mathcal{M}$ is locally isometric to $\mathbb{E}^{m+1}(0)\times{\mathcal{S}^m(4)}$ for $m>1$ and flat for $m=1$ and $\mathcal{M}$ is $\ast$-Ricci flat. Again, substituting $Ric^{\ast}=0$ and $\Omega+\frac{p}{2}+\frac{1}{2m+1}=0$ in (\ref{1.5}) infers that $\nabla^2f=0\implies\Delta{f}=0,$ which eventually implies that the function $f$ is harmonic. This completes the proof.
\end{proof}


\section{Example} 
Here we give an example of a $\ast$-conformal Einstein soliton on a $3$-dimensional $\mathcal{N}(1-\delta^2)$-contact metric manifold $\mathcal{M}$ as constructed in \cite{7}. In this example we can calculate the components of $\ast$-Ricci tensor as follows

\begin{align*} Ric^{\ast}(\mathsf{e}_1,\mathsf{e}_1)=0,\hspace{0.5cm} Ric^{\ast}(\mathsf{e}_2,\mathsf{e}_2)=Ric^{\ast}(\mathsf{e}_3,\mathsf{e}_3)=-(1-\delta^2).
\end{align*} Therefore in view of the above values of $\ast$-Ricci tensor, we have \begin{align*} r^{\ast}=-2(1-\delta^2).\end{align*} Also we can easily calculate Lie derivative of $g$ along $\mathsf{e}_1$ as
\begin{align*}(\pounds_{\mathsf{e}_1}g)(\mathcal{X},\mathcal{Y})=0\hspace{0.5cm} \forall\hspace{0.2cm} \mathcal{X},\mathcal{Y}\in\{\mathsf{e}_i: i=1,2,3\}.\end{align*}

For $\delta=1$, the curvature tensor $\mathcal{R}$ vanishes and also $Ric^{\ast}=0$. Now tracing the equation (\ref{1.4}) we get $\Omega=-\frac{p}{2}-\frac{1}{3}.$ Thus for this value of $\Omega$ the data $(g,\mathsf{e}_1,\Omega)$ defines a $\ast$-conformal Einstein soliton on this $3$-dimensional $\mathcal{N}(0)$-contact metric manifold $\mathcal{M}$. Moreover we can easily see that $\mathsf{e}_1$ is a Killing vector field and hence the Theorem \ref{thm3.3} and also the Remark \ref{rem3.5} are verified.\\\\
Again if $\mathsf{e}_1=\nabla{f}$, for some real valued smooth function $f$ defined on $\mathcal{M}$, then  from (\ref{1.5}) and considering $\delta=1$ we obtain 
\begin{align*} \Delta(f)=-3(\Omega+\frac{p}{2}+\frac{1}{3}),\end{align*} which is a Poisson equation. Now if $\Omega=-\frac{p}{2}-\frac{1}{3}$ then we have \begin{align*} \Delta(f)=0,\end{align*} which implies the function $f$ is harmonic. Hence the Theorem \ref{thm4.2} is also verified.\\\\\\

{\textbf{Acknowledgements:} \\\\The first author Jhantu Das is thankful to the Council of Scientific and Industrial Research, India (File no: {\bf 09/1156(0012)/2018-EMR-I}) for their financial support in the form of Senior Research Fellowship. \\\\}

\end{document}